\documentclass[a4paper,12pt]{article}
\usepackage{amsmath,amsfonts,amssymb,amsthm,amscd}
\usepackage[arrow, matrix, curve]{xy}
\usepackage[colorlinks=true,citecolor=blue]{hyperref}

\font\cmssl=cmss10 at 12 pt

\usepackage{slashed}
\usepackage[normalem]{ulem}

\newtheorem{thm}{Theorem}
\newtheorem{lem}[thm]{Lemma}
\newtheorem{prop}[thm]{Proposition}
\newtheorem{defn}[thm]{Definition}

\newtheorem{rem}[thm]{Remark}
\newtheorem{notation}[thm]{Notation}

\title{Integrability of generalized structures on odd exact  Courant algebroids using generalized connections}

\date{\today}

\author{Vicente Cort\'es,  Liana David and Marius Mirea}

\begin{document}

\maketitle

{\bf Abstract:}  Odd exact Courant algebroids constitute a simple class of transitive Courant algebroids. 
Their underlying vector bundle is of odd rank and differs from a generalized tangent  bundle by the addition of a line bundle. 
In this article, we study natural analogues of almost complex and almost  pseudo-Hermitian structures on such Courant algebroids, 
which are called $B_n$-generalized almost complex/pseudo-Hermitian structures. The corresponding integrable structures are know as $B_n$-generalized complex
structures and $B_n$-generalized pseudo-K\"ahler structures, respectively.  We  characterize the integrability of 
$B_n$-generalized almost complex/pseudo-Hermitian structures 
on odd exact Courant algebroids 
in terms of  existence of adapted generalized connections. 
We describe the  affine spaces of adapted generalized connections for such  integrable generalized structures.

\medskip\noindent
{\it MSc classification:} 53D18 (Generalized geometry a la Hitchin)


\section{Introduction}

An important feature of integrable geometric  structures  on a manifold 
(like a complex or a symplectic structure)  is the existence of an adapted connection, that is, a  torsion-free connection  which preserves the structure in question. 
For  structures which involve a metric (e.g.\   K\"{a}hler or  hyper-K\"{a}hler structures) the integrability is usually defined by the condition that the Levi-Civita  connection is an adapted connection.  

Using these ideas as a guiding principle, a characterization for the  integrability of 
various generalized structures 
on a Courant algebroid of even rank, in terms of  adapted generalized connections, was developed by the first two authors in \cite{moscow}.\ 

In this paper we approach the analogous question  in the setting of odd exact Courant algebroids. Such Courant algebroids   were introduced in \cite{rubio} as  a counter-part  of exact Courant algebroids obtained by replacing the structure group $\mathrm{O}(n,n)$ (Cartan type $D_n$) by $\mathrm{O}(n+1,n)$ (type $B_n$) and 
may be seen as one of the simplest classes of transitive Courant algebroids of odd rank. 
Geometric structures on them are usually referred to as  structures of type $B_{n}$ and play the role of CR  and related structures in generalized geometry.  
$B_{n}$-generalized complex  and  $B_{n}$-generalized pseudo-K\"{a}hler structures  were introduced 
in \cite{rubio} and \cite{moscow} respectively and showed to share common  features with the corresponding structures on exact Courant algebroids. 
For example, they admit a characterization in terms of spinors and, alternatively, in terms of ordinary tensors on the base manifold (usually referred to as components). 
Specifically, $B_{n}$-generalized complex structures were described in \cite{rubio} in terms of spinors, while 
$B_{n}$-generalized pseudo-K\"{a}hler structures were described in \cite{moscow}  in terms of components, in analogy  with the description of generalized K\"{a}hler   structures
on an exact Courant algebroid over a manifold $M$  in terms of bi-Hermitian data on $M$ (see \cite{gualtieri-thesis}).\

Our  main results in this paper (see Theorems  \ref{complex-conn-cor} and \ref{GK:thm}) provide  characterizations for the  integrability of  $B_{n}$-generalized almost complex 
and almost pseudo-Hermitian  structures  in terms of existence of adapted generalized connections.
In Theorem \ref{adapted}  we  describe the affine spaces of adapted  generalized connections for such generalized structures.  For completeness, 
in Theorem~\ref{ordinaryGK:thm}  we give a similar description for the affine spaces of adapted  generalized connections for generalized complex and generalized K\"ahler structures on arbitrary Courant algebroids  of even rank. Note that this does not include the case of odd exact Courant algebroids, since the latter have odd rank  and 
do not admit any generalized almost complex structures   
but only $B_n$-generalized almost complex structures.

One feature which distinguishes generalized geometry from classical geometry is the non-uniqueness of Levi-Civita  generalized connections  
for a given generalized metric.  This  is  also reflected by the non-uniqueness of adapted  generalized connections for $B_{n}$-generalized pseudo-K\"{a}hler structures, according to Theorem \ref{adapted} ii).

\bigskip

{\bf Acknowledgements.} Research of VC is funded by the Deutsche\linebreak Forschungsgemeinschaft 
(DFG, German Research Foundation) under Germany's Excellence Strategy, EXC 2121 ``Quantum Universe,'' 390833306 and under -- SFB-Gesch\"aftszeichen 1624 -- Projektnummer 506632645.

\section{Preliminary material}

This section is intended to fix notation.   We recall various  facts we need on 
generalized connections  (see e.g.\ \cite{garcia}) 
and geometric structures on odd exact Courant algeboids (see  \cite{JGA,rubio} for more details).

Let $(E, \langle \cdot , \cdot \rangle , \pi , [\cdot , \cdot ])$ be an odd exact Courant algebroid over a manifold $M$ of dimension $n$, with scalar product
$\langle \cdot , \cdot \rangle$ of signature $(n+1,n)$, anchor $\pi : E \rightarrow TM$ and  ($\mathbb{R}$-bilinear)  Dorfman bracket $[\cdot , \cdot ].$ 
We often identify $E$ with $E^{*}$ using $\langle \cdot , \cdot \rangle .$ Recall that 
\begin{align} 
\nonumber& [ u, [v,w]] = [[u,v], w] + [v, [u,w]]\\
\nonumber& [ u, f v] = f [u, v] + \pi (u) (f) v\\
\nonumber& \pi (u) \langle v, w\rangle = \langle  [u, v], w\rangle + \langle  v, [u, w]\rangle\\
\label{courant}& \langle [ u, v] + [v, u] , w\rangle = \pi (w) \langle u, v\rangle
\end{align}
for any $u, v,w\in \Gamma (E)$ and $f\in C^{\infty}(M)$.
A generalized connection on $E$ is an $\mathbb{R}$-linear  map 
 $$
  D: \Gamma (E) \rightarrow \Gamma (E^{*}\otimes E),\ u\mapsto  Du
  $$
  which satisfies 
  \begin{align}
 \nonumber&  D_{v} ( fu) = \pi (v)(f) u + f D_{v} u\\  
  \label{con-D}&  \pi (w) \langle u, v\rangle = \langle D_{w} u, v\rangle + \langle u, D_{w} v\rangle
  \end{align}
for any $u, v, w\in \Gamma (E)$ 
and $f\in C^{\infty}(M).$  Its torsion is the  $3$-form
$T^{D}\in \Gamma ( \Lambda^{3} E^{*})$ defined  by
\begin{equation}
T^{D}(u, v, w) = \langle D_{u} v - D_{v} u - [u, v] , w\rangle + \langle D_{w} u, v\rangle , \ \forall u, v, w\in \Gamma (E). 
\end{equation}
We extend $D$  to an $E$-connection on the tensor algebra of $E$, in the usual way.  For a section $u\in \Gamma (E)$ and endomorphism $A\in \Gamma (\mathrm{End}\, E)$, we denote by 
$\mathbf L_{u} A\in \Gamma (\mathrm{End}\, E)$
the Dorfman-Lie derivative of $A$ in the direction of $u$, defined by
$$
(\mathbf L_{u} A)v := [ u, Av] - A [u, v],\ \forall u, v\in \Gamma (E). 
$$ 
A {\cmssl generalized  (admissible) metric} on 
$E$ is a subbundle  $E_{-}\subset E$   such that 
$\langle \cdot , \cdot\rangle\vert_{E_{-}}$ is non-degenerate and 
$\pi \vert_{E_{-}} : E_{-} \rightarrow TM$ is an isomorphism 
(see  \cite{baraglia}).  
Let  $G^{\mathrm{end}}$ be the orthogonal  automorphism  of $E$
defined by $G^{\mathrm{end}}\vert_{E_{\pm}} = \pm \mathrm{Id}$,
where $E_{+} := (E_{-})^{\perp}$ is the orthogonal complement of $E_{-}$ with respect to $\langle \cdot ,  \cdot \rangle .$ 
Then $G(u, v) := \langle G^{\mathrm{end}} u, v\rangle$ is a non-degenerate metric on $E$. We shall often refer to $G$ 
(rather than $E_{-}$) 
as a generalized   metric  on $E$.\ 

A {\cmssl $B_{n}$-generalized almost complex structure}  on $E$ 
is a 
complex isotropic  rank $n$ subbundle $L \subset  E^{\mathbb{C}}$ such that $L \cap \bar{L} =0$.
The orthogonal complement  $U_{\mathbb{C}}$ of  $L \oplus \bar{L}\subset E^{\mathbb{C}}$ is of  rank one, generated by a real section $u_{0} \in \Gamma (E)$ normalized by  $\langle u_{0}, u_{0} \rangle = (-1)^{n}$ and unique up to multiplication by $\pm 1.$ 
Since $\langle u_{0}, u_{0} \rangle$ is constant,
$[u_{0} , u_{0} ]=0$  from the fourth relation in  (\ref{courant}).
Let $\mathcal F\in\Gamma (  \mathrm{End}\,  E)$ be the  endomorphism 
with $i$-eigenbundle $L$, $-i$ eigenbundle $\bar{L}$ and $\mathcal F u_{0} =0.$   
It is $\langle \cdot , \cdot \rangle$-skew-symmetric,  satisfies 
\begin{equation}\label{endo-f}
\mathcal F^{2} = -\mathrm{Id}
+ (-1)^{n}  \langle \cdot , u_{0} \rangle u_{0}
\end{equation}
and is of rank $2n$.
Conversely, any $\langle \cdot , \cdot \rangle$-skew-symmetric endomorphism  $\mathcal F\in\Gamma ( \mathrm{End}\, E)$ of rank $2n$, which satisfies (\ref{endo-f}) 
for a section  $u_{0} \in \Gamma (E)$ with   $\langle u_{0}, u_{0} \rangle = (-1)^{n}$
determines a $B_{n}$-generalized almost complex structure $L$ on $E$, defined as the $i$-eigenbundle of $\mathcal F .$ 
We often refer to $\mathcal F$ (rather than $L$)  as a $B_{n}$-generalized almost complex structure.
We say that $\mathcal F$ is integrable (or, is a {\cmssl $B_{n}$-generalized complex structure}) if 
$\Gamma (L)$ is closed under the Dorfman bracket of $E$. If $\mathcal F$ is integrable, then 
$\mathbf  L_{u_{0}}\mathcal F =0$.\

A  {\cmssl $B_{n}$-generalized almost pseudo-Hermitian   structure}  on $E$  
is a 
generalized metric $G$  together with a $B_{n}$-generalized almost complex structure
 $\mathcal F$ such that $G^{\mathrm{end}} \mathcal F = \mathcal F G^{\mathrm{end}}.$
Then   
$G^{\mathrm{end}} (u_{0}) = (-1)^{n} u_{0}$ and 
$G^{\mathrm{end}} \mathcal F$ is
 also a  $B_{n}$-generalized almost complex structure. Both $\mathcal F$ and $ G^{\mathrm{end}} \mathcal F$ are 
 $G$-skew-symmetric.  We say that $( G, \mathcal F )$ 
  is   {\cmssl  integrable}  (or is a {\cmssl $B_{n}$-generalized pseudo-K\"{a}hler structure}) 
 if $\mathcal F$ and $G^{\mathrm{end}}\mathcal F$ are (integrable) $B_{n}$-generalized complex  structures.
 If $(G, \mathcal F )$ is a $B_{n}$-generalized pseudo-K\"{a}hler structure then  $\mathbf L_{u_{0}} G^{\mathrm{end}} =0.$ 
    
 \section{Integrability of  $B_{n}$-generalized almost complex structures}

Let $(E, \langle \cdot , \cdot \rangle , \pi ,
[\cdot , \cdot ] )$ be an odd exact Courant algebroid over an $n$-dimensional manifold $M$.
The next proposition is  a rewriting of Proposition 4.21  of \cite{rubio} using the Dorfman bracket.
        
\begin{prop}\label{nij-prop}
Let $\mathcal F$ be a $B_{n}$-generalized almost complex structure on $E$  and $U:= \mathrm{Ker}\, \mathcal F .$ 
Then $\mathcal F$ is integrable if and only if
its Nijenhuis tensor 
$$
N_{\mathcal F}: U^{\perp}\times U^{\perp} \rightarrow E
$$
defined by 
\begin{equation}\label{ec1}
N_{\mathcal F} (u, v):= [ \mathcal F u, \mathcal F v] - [u, v] - \mathcal F \left( [ \mathcal F u, v] + [ u, \mathcal F v] \right) ,
\end{equation}
for any $u, v\in \Gamma (U^{\perp})$, 
vanishes. 
\end{prop}

\begin{proof} 
From  \cite[Section 4.1.4]{rubio},
$\mathcal F$ is integrable if and only if for any $u, v, w\in \Gamma (E)$,
\begin{equation}\label{ec0}
\langle [ \mathcal F u - i \mathcal F^{2} u, \mathcal F v - i \mathcal F^{2} v]_{\mathcal C}, w- i \mathcal F w\rangle=0. 
\end{equation}
Using the fourth relation in (\ref{courant}) and that  $L$ is isotropic, we can replace in the above relation the Courant bracket 
$[\cdot , \cdot ]_{\mathcal C}$ 
(i.e.\  the skew-symmetric part of $[\cdot , \cdot ]$) 
with the Dorfman bracket $[\cdot , \cdot ]$. 
Relation (\ref{ec0}) is obviously satisfied when either $u$ or $v$ is a section of $U$.
Since $\mathcal F^{2} = -\mathrm{Id}$ on $U^{\perp}$, 
we obtain that (\ref{ec0})
is satisfied if and only if for any $u, v\in \Gamma (U^{\perp})$ and $w\in \Gamma (E)$,
\begin{equation}\label{dorf-int}
\langle [ u - i \mathcal F u,  v - i \mathcal F  v], w- i \mathcal F w\rangle=0  
\end{equation}
which is equivalent to (\ref{ec1}).
\end{proof}

\begin{thm}\label{complex-conn} Let  $\mathcal F \in\Gamma ( \mathrm{End }\, E)$  be a $B_{n}$-generalized almost complex structure  on 
$E$, $U:=\mathrm{Ker}\, \mathcal F$ 
and   
$u_{0}\in \Gamma ( U )$ such that $\langle u_{0}, u_{0} \rangle = (-1)^{n}$. 
There is a generalized connection $D$ which preserves $\mathcal F$, i.e.\ $D \mathcal F =0$,
and whose torsion satisfies
\begin{align} 
\nonumber& T^{D} (u, v, w) = \frac{1}{4} \langle N_{\mathcal F}(u, v), w\rangle ,\ \forall u, v, w\in \Gamma (U^{\perp})\\
\label{torsion-c}& T^{D} (u, v, u_{0} )=\frac{1}{2} \langle  ({\mathbf L}_{u_{0}} \mathcal F ) u,  \mathcal Fv\rangle ,\
\forall u, v\in \Gamma (E).
\end{align}
\end{thm}

We divide the proof of the above theorem into several lemmas. Let $\mathcal F$ be a $B_{n}$-generalized
almost complex structure on $E$, as in Theorem \ref{complex-conn}.

\begin{lem}\label{pr-add} 
i) There is a torsion-free generalized
connection  $D$   on $E$ such that  $Du_{0}=0.$\ 

ii) Let $D$ be a generalized connection on $E$ such that $Du_{0} =0.$  Then
\begin{align}
\nonumber& \langle N_{\mathcal F}(u, v),w\rangle =\\
\nonumber&  - T^{D} (\mathcal F u,\mathcal F v, w) + T^{D} (u, v,w) 
- T^{D}(\mathcal F u, v, \mathcal F w) - T^{D} (u, \mathcal F v, \mathcal F w)\\
\nonumber& + \langle (D_{\mathcal F u} \mathcal F )v, w\rangle -  \langle (D_{\mathcal F v} \mathcal F )u, w\rangle 
+\langle \mathcal F (D_{v} \mathcal F )u, w\rangle \\
\label{rel-n-c} & - \langle \mathcal F (D_{u} \mathcal F)v, w\rangle - \langle \mathcal F (  D_{w} \mathcal F)u,v\rangle
+\langle (D_{\mathcal F w} \mathcal F )u,v\rangle ,
\end{align}
for any $u, v\in \Gamma (U^{\perp})$  and $w\in \Gamma (E).$
\end{lem}

\begin{proof} i) Let $D$ be a  torsion-free generalized connection. The generalized connection $D^{(1)}_{u} := D_{u}
+\eta_{u}$ where $\eta_{u} := (-1)^{n+1} u_{0}\wedge D_{u} u_{0}$ for $u\in \Gamma ( U^{\perp})$
and $\eta_{u_{0}}:= - 2 (Du_{0})^{\mathrm{sk}}$ where 
$(Du_{0})^{\mathrm{sk}}$ is the skew-symmetric part  of  the\linebreak[4] endomorphism 
$Du_{0}\in \Gamma (\mathrm{End}\, E)$,
is also torsion-free and moreover preserves $u_{0}.$ Here we have 
identified $\Lambda^2E$ with $\mathfrak{so}(E)$ by means of the scalar product, using the following convention $(u\wedge v)(w) = \langle u,w\rangle v-\langle v,w\rangle u$.

 ii)  The claim follows 
 by applying the formula 
 $$
\langle [ u, v], w\rangle = \langle D_{u}v - D_{v} u ,w\rangle +\langle D_{w} u, v\rangle - T^{D}(u, v, w)
$$
 for various terms in  $\langle N_{\mathcal F}(u, v), w\rangle$ (see relation (\ref{ec1}))
 and using   (\ref{endo-f}), the skew-symmetry of $\mathcal F$  and that $D$ preserves $U^{\perp}$ 
(which follows from $Du_{0} =0$).
\end{proof}

For a generalized connection $D$ on $E$ we define 
\begin{equation}\label{tors-anti}
 (T^{D})^{\mathcal F , -}(u, v, u_{0}  ): =\frac{1}{2} \left( T^{D} ( \mathcal F u, \mathcal F v, u_{0}) -
T^{D} (u, v, u_{0}) \right).
\end{equation}

\begin{lem}\label{pr-add-1}
Any generalized connection $D$ which preserves $\mathcal F$  
preserves also $u_{0}$. Its torsion satisfies 
\begin{equation}\label{tors-j-anti}
 (T^{D})^{\mathcal F , -}(u, v, u_{0}  )=-\frac{1}{2}   \langle  ( {\mathbf L}_{u_{0}} \mathcal F ) u,  \mathcal F v\rangle ,\
\forall u, v\in \Gamma (E).
\end{equation} 
\end{lem}

\begin{proof}
From $D\mathcal F =0$ we obtain that  $D u_{0} = \lambda \otimes u_{0}$ for some $\lambda \in \Gamma (E^*)$, since $u_0$ spans the kernel of $\mathcal F$. Taking the scalar product of this relation with $u_{0}$ and using that $\langle u_{0}, u_{0}\rangle
= (-1)^{n}$ we obtain that $\lambda =0$, i.e.\ $Du_{0} =0.$
When $u$ or $v$ is a section of $U$, both sides of  (\ref{tors-j-anti}) vanish. Assume now that $u, v\in \Gamma (U^{\perp}).$ 
Since $D$ preserves $U^{\perp}$, 
\begin{align}
\nonumber T^{D} (u, v, u_{0})&  =  - \langle [u, v], u_{0} \rangle + \langle D_{u_{0}} u, v\rangle\\
\nonumber& = \langle v, [u, u_{0}]\rangle  +  \langle D_{u_{0}} u, v\rangle\\
\label{t-d-u0}&  = \langle D_{u_{0}} u - [u_{0}, u], v\rangle ,
\end{align}
where we used the definition of $T^{D}$ and the third, respectively  
fourth relation in (\ref{courant}).
Relation  (\ref{tors-j-anti})  with $u, v\in \Gamma (U^{\perp})$ follows from (\ref{t-d-u0}) applied to both terms in the right hand side of 
(\ref{tors-anti}).
\end{proof}

With the above  preliminary lemmas, we  now construct a generalized connection 
required by Theorem
\ref{complex-conn}. 
Let  $D$ be a torsion-free generalized  connection 
such that $Du_{0} =0$ (which exists by  Lemma
\ref{pr-add} i)). For any $u\in \Gamma (E)$, let  $A_{u}\in \Gamma ( \mathrm{End}\, (U^{\perp}))$ defined by 
\begin{equation}
\langle A_{u}v, w\rangle  =\frac{1}{2} \left( \langle (D_{v} \mathcal F)u, w\rangle + \langle (D_{w} \mathcal F )u, v
\rangle \right),\ \forall v, w\in \Gamma (U^{\perp})
\end{equation}
and extend it to a (symmetric)  endomorphism of $E$ (also denoted by $A_{u}$) which acts trivially on $U$.
We denote by
$$
\{ A_{u}, \mathcal F \} = A_{u}\circ \mathcal F + \mathcal F \circ A_{u}
$$
the anti-commutator of $A_{u}$ and $\mathcal F .$
We will prove that the generalized connection $\tilde{D}$ defined in the next lemma satisfies the conditions required by Theorem \ref{complex-conn}. 
The next lemma is the analogue of Lemma~26 of \cite{moscow}. 

\begin{lem}\label{l1} For any $u, v\in \Gamma (E)$, define $\tilde{D}_{u}v\in \Gamma (E)$ by  
\begin{equation}
\langle \tilde{D}_{u} v, w\rangle = \langle D_{u} v, w\rangle -\frac{1}{4} \langle \{ A_{u}, \mathcal F \} v, w\rangle 
-\frac{1}{2} \langle \mathcal F (D_{u}\mathcal F )v, w\rangle  ,\ \forall w\in \Gamma (E).
\end{equation}
Then $\tilde{D}$ is a generalized connection which preserves $\mathcal F .$
\end{lem}  

 \begin{proof}
From  (\ref{endo-f}) and $D u_{0} =0$, we obtain that
$D_{u} \mathcal F$ and $\mathcal F$ anti-commute, for any $u\in \Gamma (E).$ 
The endomorphisms $\mathcal F$ and  $D_{u}\mathcal F$ are skew-symmetric, 
while $A_{u}$ is symmetric.  We obtain
that   $\{ A_{u}, \mathcal F \}$ and 
$\mathcal F D_{u}\mathcal F $ are both skew-symmetric and hence $\tilde{D}$ is a generalized connection.
Moreover,  $\{ A_{u}, \mathcal F \}$ commutes with $\mathcal F$ while
$$
D^{(1)}_{u} : = D_{u}-\frac{1}{2} \mathcal F D_{u} \mathcal F
$$
is a generalized connection which preserves $\mathcal F .$
Hence  $\tilde{D}$ preserves $\mathcal F .$
\end{proof}

The next lemma concludes the proof of Theorem \ref{complex-conn}.

\begin{lem}\label{l2} The torsion of  $\tilde{D}$ satisfies relations
(\ref{torsion-c}).
\end{lem}

\begin{proof} The generalized connections $\tilde{D}$ and $D$ are related by
$\tilde{D} = D +\eta$
where
\begin{equation}\label{def-eta}
\eta_{u} (v, w):= -\frac{1}{4} \langle \{ A_{u}, \mathcal F \} v, w\rangle 
-\frac{1}{2} \langle \mathcal F (D_{u}\mathcal F )v, w\rangle  ,
\end{equation}
for any $u, v, w\in \Gamma (E)$. Note that 
\begin{equation}\label{eta-0}
\eta_{u}(u_{0}, w) =0,\   \eta_{u_{0}} (v, w) = -\frac{1}{2} \langle \mathcal F (D_{u_{0}} \mathcal F )v, w\rangle .
\end{equation}
Since $D$ is torsion-free,
\begin{equation}
T^{\tilde{D}} (v, w, u_{0}) =\eta_{v} (w, u_{0} )+  \eta_{u_{0}} (v, w) + \eta_{w} (u_{0}, v) =  -\frac{1}{2} \langle \mathcal F (D_{u_{0}} \mathcal F )v,w\rangle ,
\end{equation}
which implies that $T^{\tilde{D}} = - (T^{\tilde{D}} )^{\mathcal F , -}.$ 
The second relation  in 
(\ref{torsion-c}) 
follows from Lemmas \ref{pr-add-1} and \ref{l1}. 
The first relation in  (\ref{torsion-c})  follows from 
an argument  analogous to  Lemmas 27, 28 and 29 of \cite{moscow}.
For this reason, we only preserve the main steps:  we compute first
\begin{equation}\label{d-1}
T^{D^{(1)} } (u, v, w) =\frac{1}{2} \sum_{( u, v, w)} \gamma (u, v, \mathcal F w),
\end{equation}
for any $u, v, w\in E$, 
where  $\gamma (u, v, w) := \langle (D_{u} \mathcal F  )v, w\rangle$ and 
$\sum_{( u,v,w)}$ denotes  sum over cyclic permutations. 
Using (\ref{d-1}) and $\tilde{D}_{u} = D^{(1)}_{u} -\frac{1}{4} \{ A_{u}, \mathcal F \} $
we obtain 
\begin{equation}\label{dt}
T^{\tilde{D}} (u, v,w) =\frac{1}{4}  \sum_{(u, v,w)} \left( \gamma (\mathcal F u, v,w) +\gamma (u, v, \mathcal F w) \right) ,
\end{equation}
for any $u, v, w\in U^{\perp}.$ On the other hand, from 
 (\ref{rel-n-c}), the right hand side of (\ref{dt}) equals $\frac{1}{4} \langle N_{\mathcal F} (u, v), w\rangle$,    since $D$ is torsion-free and $Du_{0} =0.$
\end{proof}

We arrive at the main result of this section.

\begin{thm}\label{complex-conn-cor} A $B_{n}$-generalized almost complex structure $\mathcal F$ on $E$ is integrable 
if and only if there is a torsion-free generalized connection $D$ which preserves $\mathcal F .$ 
\end{thm}

\begin{proof} Assume that $\mathcal F$ is a $B_{n}$-generalized complex structure  and let $D$ be a generalized connection provided by
Theorem \ref{complex-conn}. Then $D\mathcal F =0$ and, from the integrability of $\mathcal F$,   
$D$ is torsion-free
(since $N_{\mathcal F} =0$ and  ${\mathbf L}_{u_{0}} \mathcal F =0$).
Conversely, assume that $\mathcal F$ is a $B_{n}$-generalized almost complex structure and $D$ a  torsion-free
generalized connection such that $D\mathcal F =0$. 
For any
$u, v\in \Gamma (U^{\perp})$, 
\begin{align}
\nonumber \langle N_{\mathcal F} (u, v), u_{0}\rangle &=\langle [ \mathcal F u, \mathcal F v] - [u, v],u_{0}\rangle
= \langle [u_{0}, \mathcal F u] , \mathcal F v \rangle -\langle [u_{0}, u], v\rangle \\
\label{n-f-l}& = \langle  (\mathbf L_{u_{0}} \mathcal F ) u,  \mathcal F v\rangle =0,
\end{align} 
where in the last equality we used  Lemma \ref{pr-add-1} and $T^{D}=0$.  From (\ref{n-f-l}) and  
$\langle N_{\mathcal F}(u,v), w\rangle =0$ for any $u, v,w\in \Gamma (U^{\perp})$ (as $T^{D} =0$, see Theorem \ref{complex-conn}),   we obtain $N_{\mathcal F} =0$.
\end{proof}

\section{Integrability of $B_{n}$-generalized almost pseudo-Hermitian structures}
  
 In this section we prove:

\begin{thm} \label{GK:thm}A $B_{n}$-generalized almost pseudo-Hermitian structure  $(G, \mathcal F )$  on $E$ is integrable  if and only if
there is a torsion-free generalized connection $D$ which preserves $(G, \mathcal F)$, i.e.\   $DG =0$ and $D\mathcal F =0.$ 
\end{thm}

One statement  is obvious: if there is a torsion-free generalized connection $D$ such that 
$DG =0$ and $D \mathcal F=0$
then 
$\mathcal F_{1} =\mathcal F$ and 
$\mathcal F_{2}= G^{\mathrm{end}} \mathcal F$ are integrable, from Theorem  \ref{complex-conn-cor}.
For the converse, let $(G, \mathcal F)$ be a $B_{n}$-generalized pseudo-K\"{a}hler structure and $u_{0}\in\Gamma (U)
=\Gamma ( \mathrm{Ker}\, \mathcal F )$ such that
$\langle u_{0}, u_{0}\rangle = (-1)^{n}.$ 

\begin{lem}\label{st1} There is a torsion-free generalized connection $D$ such that $DG =0$ and $Du_{0} =0.$
\end{lem} 

\begin{proof}
Start with a  Levi-Civita generalized connection $D$ of $G$,
i.e.\ a  torsion-free generalized connection which preserves $G$ (see e.g.\  \cite{garcia} for its existence). 
Let 
$D^{(1)}_{u} = D_{u} +\eta_{u}$
be the generalized connection from  the proof of Lemma~\ref{pr-add}~i),
where, we recall, 
$\eta_{u} = (-1)^{n+1} u_{0} \wedge D_{u} u_{0}$ for $u\in \Gamma ( U^{\perp})$ and
$\eta_{u_{0}} = - 2 (Du_{0})^{\mathrm{sk}}$.  Then
$D^{(1)}$ is torsion-free and 
$D^{(1)}u_{0} =0.$ We claim that $D^{(1)} G =0.$ This reduces to showing that $\eta_{u}$ is $G$-skew-symmetric,
for any $u\in\Gamma ( E).$ 
For $u := u_{0}$ this 
is equivalent to 
\begin{equation}\label{eki}
\langle D_{G^{\mathrm{end}}v} u_{0}, w\rangle + \langle D_{G^{\mathrm{end}}w} u_{0}, v\rangle
=  \langle D_{v} u_{0},G^{\mathrm{end}} w\rangle + \langle D_{w} u_{0}, G^{\mathrm{end}}v\rangle,
\end{equation}
for any $u, v\in \Gamma (E)$. Using that $T^{D} =0$,  we write
$$
\langle D_{G^{\mathrm{end}}v} u_{0}, w\rangle =\langle D_{u_{0}} (G^{\mathrm{end}} v), w\rangle
- \langle D_{w} (G^{\mathrm{end}} v), u_{0} \rangle +\langle [ G^{\mathrm{end}} v, u_{0} ], w\rangle
 $$
and similarly for  $\langle D_{G^{\mathrm{end}}w} u_{0}, v\rangle .$
Using $D G^{\mathrm{end}} =0$ and the fourth relation in \eqref{courant},  we
 obtain that  (\ref{eki}) is equivalent to 
\begin{align}
\nonumber& G( D_{u_{0}} w, v) -\langle [ u_{0}, G^{\mathrm{end}} w], v\rangle + G(D_{u_{0}} v, w)\\
\label{eki-1}& - G( D_{w} v, u_{0}) + \langle [ G^{\mathrm{end}} v, u_{0} ], w\rangle - G( D_{w} u_{0}, v)=0
\end{align}
or to 
\begin{equation}\label{eki-2}
\pi ( u_{0} ) G(v, w) -\langle [ u_{0}, G^{\mathrm{end}}w], v\rangle = \pi (w) G(u_{0}, v) -\langle 
[ G^{\mathrm{end}} v, u_{0} ], w\rangle ,
\end{equation}
where we used $D G^{\mathrm{end}} =0$ again.
But 
\begin{align*}
\langle [ G^{\mathrm{end}}v, u_{0} ], w\rangle 
\nonumber& = - \langle [ u_{0}, G^{\mathrm{end}} v], w\rangle + \pi (w) G(u_{0}, v)
\end{align*}
and the right-hand side of (\ref{eki-2})  equals $\langle [ u_{0},  G^{\mathrm{end}} v], w\rangle .$ 
Relation (\ref{eki-2}) follows now from the third relation in  (\ref{courant}), 
${\mathbf L}_{u_{0}} G^{\mathrm{end}} =0$ and the $\langle\cdot , \cdot \rangle$-symmetry of $G^{\mathrm{end}}.$ 
This proves the $G$-skew-symmetry of $\eta_{u_{0}}.$  
The $G$-skew-symmetry   of $\eta_{u}$ for $u\in \Gamma ( U^{\perp})$ follows by writing  
$$
\eta_{u} = \frac{(-1)^{n+1}}{2} \left( u_{0} \wedge D_{u} u_{0} + G^{\mathrm{end}}u_{0} \wedge G^{\mathrm{end}} D_{u} u_{0}\right)
$$
and noticing that  $u\wedge v + G^{\mathrm{end}}u\wedge G^{\mathrm{end}}v$ is $G$-skew-symmetric for any $u, v.$
\end{proof}

The next lemma  concludes the proof of Theorem \ref{GK:thm}. 
 
\begin{lem} Let $D$ be a generalized connection like in Lemma \ref{st1}. Then the generalized connection
$\tilde{D}$ defined in Lemma \ref{l1} starting from $D$ satisfies $\tilde{D}G =0$,  $\tilde{D}\mathcal F =0$
and $T^{\tilde{D}} =0.$ 
\end{lem}

\begin{proof} From Lemmas \ref{l1} and \ref{l2},  $\tilde{D} \mathcal F =0$
and $T^{\tilde{D}}=0$.
It remains to show that $\tilde{D} G=0$. 
Since $\mathcal F$ and $D_{u} \mathcal F$ are anticommuting $G$-skew-symmetric endomorphisms, 
$\mathcal F D_{u}\mathcal F$ is $G$-skew-symmetric. It remains  to show that $\{ A_{u}, \mathcal F \}$ is  $G$-skew-symmetric.
Since $\{ A_{u}, \mathcal F \}$ is  $\langle\cdot , \cdot \rangle$-skew-symmetric
(see the proof of Lemma 
\ref{l1}),  it is sufficient to prove that 
$\{ A_{u}, \mathcal F \}$  preserves $E_{\pm }$. For this, we will show that 
\begin{equation}\label{decompos-applic-0}
\langle (D_{v} \mathcal F)u, w\rangle =0,\ \forall u\in \Gamma (E),\ v\in \Gamma (E_{\pm}),\ w\in \Gamma (E_{\mp}). 
\end{equation}
(Relation  \eqref{decompos-applic-0} implies that $A_u$, and therefore $\{ A_{u}, \mathcal F \}$, preserves $E_{\pm }$). 
Let $v\in \Gamma (E_{+})$ and $w\in \Gamma (E_{-})$ (the argument for $v\in \Gamma (E_{-})$ and $w\in \Gamma (E_{+})$ 
is similar).   Relation (\ref{decompos-applic-0})
with  $u\in \Gamma (E_{+})$  follows from the fact that $D$  and $\mathcal F$ preserve $E_{+}.$
Let $u\in \Gamma (E_{-}).$ Since  $T^{D} =0$, 
\begin{equation}\label{tors-fin}
\langle D_{v} u, w\rangle = \langle D_{u}v, w\rangle -\langle D_{w} v, u\rangle +\langle [v,u],w\rangle
\end{equation}
and similarly for $u$ replaced by $\mathcal F u$, relation
(\ref{decompos-applic-0}) becomes
\begin{equation}\label{fin}
\langle [v, \mathcal F u] - \mathcal F [v,u] , w\rangle =0,
\end{equation}
for any $u\in \Gamma (E_{-}),\ v\in\Gamma (E_{+}),\ w\in \Gamma (E_{-}).$ 
Let $L$ be the $i$-eigenbundle of $\mathcal F$. Then $(E_{\pm})^{\mathbb{C}}$ decompose as
\begin{align}
\nonumber&  (E_{-})^{\mathbb{C}} = (E_{-})^{\mathbb{C}} \cap L \oplus (E_{-})^{\mathbb{C}} \cap \bar{L}\\
\nonumber& (E_{+})^{\mathbb{C}} = \mathrm{span}_{\mathbb{C}} (u_{0}) \oplus (E_{+})^{\mathbb{C}} \cap 
L \oplus (E_{+})^{\mathbb{C}} \cap \bar{L}
\end{align}
if $n = \mathrm{dim}\, M$ is even and
\begin{align}
\nonumber& (E_{-})^{\mathbb{C}} = \mathrm{span}_{\mathbb{C}} (u_{0}) \oplus (E_{-})^{\mathbb{C}} \cap 
L \oplus (E_{-})^{\mathbb{C}} \cap \bar{L}\\
\nonumber&  (E_{+})^{\mathbb{C}} = (E_{+})^{\mathbb{C}} \cap L \oplus (E_{+})^{\mathbb{C}} \cap \bar{L}
\end{align}
if $n$ is odd. 
Relation  (\ref{fin}) can be proved  by taking $u$, $v$, $w$  sections of various 
components in the above decompositions.  For example,  if $n$ is even and $v = u_{0}$, 
relation (\ref{fin}) follows from $\mathbf L_{u_{0}} \mathcal F =0$ 
(which holds as $\mathcal F$ is integrable). The other cases can be treated similarly and use the integrability  of both $\mathcal F$ and
$G^{\mathrm{end}} \mathcal F .$ 
 A similar argument was employed in  \cite[Lemma 34]{moscow}.
\end{proof}

\section{The space of adapted  generalized connections}

Let $E$ be an odd exact Courant algebroid over an $n$-dimensional  manifold $M$.

 \begin{defn}
i) Let $\mathcal Q$ be a set of tensor fields on $E$. A generalized connection $D$ on $E$ is called {\cmssl adapted to $\mathcal Q$}
if it is torsion-free and $DA=0$ for any $A\in \mathcal Q $.\

ii) A generalized connection adapted to a $B_{n}$-generalized  complex (respectively to a $B_{n}$-generalized pseudo-K\"{a}hler)  structure 
is  called  a {\cmssl   $B_{n}$-complex}  (respectively  {\cmssl  $B_{n}$-K\"{a}hler})  {\cmssl generalized connection}.
\end{defn}

\begin{notation}{\rm   
In the next theorem we  denote by   $\mathrm{sk} : S^{2} (E^{\mathbb{C}} )^{*} \otimes (E^{\mathbb{C}} )^{*} \rightarrow 
(E^{\mathbb{C}})^{*} \otimes \Lambda^{2}  (E^{\mathbb{C}})^{*}$   the map
\begin{equation}\label{skew-def-1}
(\mathrm{sk}\, \sigma )(u, v,w):= \sigma (u, v,w) -\sigma (u, w,v),\ \forall u, v,w.
\end{equation}
The sections   $\eta$, $\eta_{\pm}$
are considered as sections of $S^{2} (E^{\mathbb{C}} )^{*}\otimes (E^{\mathbb{C}})^{*}$ by means of the natural inclusions 
of $S^{2} L^{*}\otimes 
\bar{L}^{*}$ and $S^{2} E_{\pm , \mathcal F}^{*}\otimes 
\bar{E}_{\pm , \mathcal F}^{*}$  in $S^{2} (E^{\mathbb{C}})^{*}\otimes 
(E^{\mathbb{C}})^{*}$   defined by  the decompositions 
$E^{\mathbb{C}}  = L\oplus \bar{L} \oplus  U^{\mathbb{C}}$ 
and  $E^{\mathbb{C}}  =  E_{+, \mathcal F} \oplus \bar{E}_{+, \mathcal F} \oplus 
E_{-, \mathcal F} \oplus \bar{E}_{-, \mathcal F} \oplus U^{\mathbb{C}}$ 
induced by  $\mathcal F$ and $(G, \mathcal F )$  respectively,
where $L$ is the $i$-eigenbundle of $\mathcal F$,  $U= \mathrm{Ker}\, \mathcal F$,  $E_{\pm}$ are the $\pm1$ eigenbundles of $G^{\mathrm{end}}$
and   
$E_{\pm , \mathcal F} = E_{\pm}^{\mathbb{C}} \cap L$.
For $\sigma \in \Gamma  ( 
(E^{\mathbb{C}})^{*} \otimes \Lambda^{2}  (E^{\mathbb{C}})^{*})$, $\mathrm{Re}\, \sigma \in \Gamma (
E^*  \otimes \Lambda^{2}  (E^{*}))$ maps $u, v, w\in \Gamma (E)$ to the real part of $\sigma (u, v, w).$ }
\end{notation}

The next theorem is our main result in this section.

\begin{thm}\label{adapted} i) Let $\mathcal F$ be a $B_{n}$-generalized  complex structure on $E$ with $i$-eigenbundle $L$.   
Any two $B_{n}$-complex  generalized connections  are related by  
\begin{equation}
\tilde{D} = D +\mathrm{Re}\,  ( \mathrm{sk}\, \eta ),
\end{equation}
where $\eta \in \Gamma ( S^{2} L^{*}\otimes 
\bar{L}^{*})$.

ii) Let $(G, \mathcal F )$ be a $B_{n}$-generalized pseudo-K\"{a}hler structure on $E$.
Any  two $B_{n}$-K\"{a}hler  generalized connections  are related by
\begin{equation}
\tilde{D} = D +\mathrm{Re}\left(   \mathrm{sk}\,  (\eta_{+}  + \eta_{-} ) \right),
\end{equation}
where  $\eta_{\pm} \in \Gamma ( S^{2} E_{\pm , \mathcal F}^{*}\otimes 
\bar{E}_{\pm , \mathcal F}^{*})$.
\end{thm} 

In order to prove Theorem \ref{adapted}, we start with several general considerations. 
Consider a reduction of the structure group of  $E$ 
to a subgroup $H\subset \mathrm{O}(n+1,n)$, defined as the stabilizer of a system of tensors.
This system of tensors on $\mathbb{R}^{2n+1}$ gives rise to a system of sections  $\mathcal Q$ of the tensor algebra of $E$.
We denote by 
\[ \mathfrak{h}(E_p)\cong \mathfrak{h}= \mathrm{Lie}\, H  \subset \mathfrak{so}(n+1,n)\] 
the subalgebra of $\mathfrak{so}(E_p)$ stabilizing the sections of $\mathcal Q$  at the 
point $p\in M$ and by $\mathfrak{h}(E)\subset \mathfrak{so}(E)$ the corresponding subbundle.
From the theory developed in \cite{moscow},  the space 
of adapted generalized connections to a  $H$-structure 
is either empty or  is an affine space 
modelled on the space of sections of the first generalized prolongation 
\[ \mathfrak{h}(E)^{\langle 1\rangle} := \{ \alpha \in E^*\otimes \mathfrak{h}(E) \mid \partial \alpha =0\} . \]
Here 
$\partial \alpha \in \Lambda^3 E^*$ is defined for every $\alpha \in E^*\otimes \mathfrak{so}(E)$ by 
\[ (\partial \alpha) (u,v,w) := \sum_{(u,v,w)} \alpha (u,v,w),\] 
where $(u,v,w)$ indicates cyclic summation
and $\alpha$ is identified with an element of $E^*\otimes \Lambda^2E^*$ using the scalar product. 
In order to prove Theorem \ref{adapted}  it is  therefore sufficient  to determine the first  generalized prolongations 
for the Lie subalgebras of the structure groups of   $B_{n}$-generalized complex  and $B_{n}$-generalized pseudo-K\"{a}hler structures
(the existence of an adapted connection was proved in Theorems  \ref{complex-conn-cor} and \ref{GK:thm}).
This will be done in the next two subsections.

\begin{rem}{\rm The first generalized prolongation for subalgebras of $\mathfrak{so}(n)$  
has occurred in the classification of holonomy algebras of Lorentzian manifolds as 
a space of weak curvature endomorphisms \cite{LeistnerJDG}.  In the case of 
subalgebras of $\mathfrak{u}(m)\subset \mathfrak{so}(2m)$ it reduces to the ordinary first prolongation
of the corresponding complex Lie subalgebra of $\mathfrak{gl}(m,\mathbb{C})$. This is essentially 
the content of \cite[Proposition 3.1]{LeistnerJDG}, where on the left-hand side 
of the statement the space $B_h(\mathfrak{g}_0^\mathbb{C})$ should be replaced by 
$B_h(\mathfrak{g}_0)$ for the injectivity argument in the proof to work, compare Proposition \ref{um_1um_2:prop} below. }
\end{rem}

By a similar (slightly simpler) proof as that of Theorem \ref{adapted}, we  also obtain the following theorem. 
We emphasize that existence of an adapted generalized connection for any generalized complex or generalized K\"ahler structure was proven in \cite{moscow}, ensuring that the affine spaces in Theorem~\ref{ordinaryGK:thm} are non-empty. 
\begin{thm}\label{ordinaryGK:thm}   Let $E$ be a Courant algebroid over $M$  of even rank.\begin{enumerate}
\item[(i)] The space of generalized connections adapted to a generalized complex structure $\mathcal J$ 
on $E$ is an affine space modelled on the vector space 
\[ \{ \mathrm{Re}\,  ( \mathrm{sk}\, \eta ) \mid \eta \in \Gamma ( S^{2} L^{*}\otimes 
\bar{L}^{*}) \} = \Gamma (\mathfrak{u}(E)^{\langle 1\rangle }), \] 
where $L$ is the $(1,0)$-bundle  of $\mathcal J$ and $\mathfrak{u}(E) \subset \mathfrak{so}(E)$ is the bundle of Lie subalgebras which stabilizes $\mathcal J$. 
\item[(ii)] The space of generalized connections adapted to a  generalized K\"ahler structure $(G,\mathcal J)$ on $E$ is an affine space modelled on the vector space 
\[ \{ \mathrm{Re}\left(   \mathrm{sk}\,  (\eta_{+}  + \eta_{-} ) \right)\mid \eta_\pm \in \Gamma ( S^{2} (E^{\mathbb{C}}_\pm \cap L)^{*}\otimes 
(E^{\mathbb{C}}_\pm \cap\bar{L})^{*}) \} = \Gamma (\mathfrak{u}(E)_{G}^{\langle 1\rangle }), \]
where  $L$ is the $(1,0)$-bundle of $\mathcal J$, $E_{\pm}$ are the $\pm1$-eigenbundles of $G^{\mathrm{end}}$, 
$\mathfrak{u}(E)_{G} =\mathfrak{u}(E_+)\oplus \mathfrak{u}(E_-)\subset \mathfrak{u}(E)$ denotes the stabilizer of $G$ and $\mathfrak{u}(E)_{G}^{\langle 1\rangle }:= (\mathfrak{u}(E)_{G})^{\langle 1\rangle }= \mathfrak{u}(E_+)^{\langle 1\rangle }\oplus  \mathfrak{u}(E_-)^{\langle 1\rangle }$. 
\end{enumerate} 
\end{thm}
Note that fiber-wise $\mathfrak{u}(E)_p\cong \mathfrak{u}(k,\ell )$ ($p\in M$), where $(2k,2\ell)$ is the signature of the scalar product on the 
Courant algebroid $E$. Similarly, $\mathfrak{u}(E_\pm)_G|_p\cong 
\mathfrak{u}(k^\pm ,\ell^\pm)$, where $(2k^\pm ,2\ell^\pm)$ is the signature of $E_\pm\subset (E,\langle \cdot , \cdot \rangle )$.
In particular, $G$ is positive definite if and only if $k^-=\ell^+=0$.

\subsection{$B_{n}$-complex generalized connections}
The structure group of a $B_{n}$-generalized almost complex structure $\mathcal{F}$ on $E$  is 
\[ H = \mathrm{O}(1) \times \mathrm{U}(m_1,m_2),\quad (m_1,m_2) = \begin{cases} (m,m)\quad \mbox{if}\quad n=2m\\
(m+1,m) \quad\mbox{if}\quad n=2m+1.\end{cases}\]
So 
\[ \mathfrak{h} = \mathfrak{u}(m_1,m_2) \subset \mathfrak{so}(2m_1,2m_2)\subset \mathfrak{so}(V),\quad V= \mathbb{R}^{n+1,n}= \mathbb{R}^{2m_1,2m_2} \oplus \mathbb{R},\]
where $\mathbb{R}$ denotes the orthogonal complement of $\mathbb{R}^{2m_1,2m_2}$, which is a positive definite line if $n$ is even and negative definite otherwise. 
The vector space $V$ is endowed with a skew-symmetric operator, which we also denote by  $\mathcal F$, that  induces a complex structure on $\mathbb{R}^{2m_1,2m_2}$ and vanishes on $\mathbb{R}$.

\begin{prop} \label{1stprol:prop}The generalized first prolongation of the Lie algebra 
\[  \mathfrak{h} = \mathfrak{u}(m_1,m_2) \subset \mathfrak{so}(V)\] 
is given by 
\[  \mathfrak{h}^{\langle 1\rangle }  = \{ \mathrm{Re} \left( \mathrm{sk} (\sigma) \right) \mid  \sigma \in S^2V_\mathcal F^* \otimes \bar{V}_\mathcal F^*\}\cong  S^2\mathbb{C}^n \otimes \mathbb{C}^n,\]
where $V_\mathcal F\cong \mathbb{C}^n$ is the $i$-eigenspace of $\mathcal F$.
Equivalently, 
\[  \mathfrak{h}^{\langle 1\rangle }  = \mathrm{span} \{ \eta + \bar{\eta} \mid \eta = \alpha \otimes  (\beta \wedge \bar{\gamma} )+ \beta \otimes  (\alpha \wedge \bar{\gamma}),\quad 
\alpha, \beta , \gamma \in V_\mathcal F^* \}.\]
\end{prop}  
\begin{proof}
The generalized first prolongation consists of the real elements in the kernel of 
\begin{equation}\label{map-c}  
\partial : V^*_\mathbb{C}\otimes V_\mathcal F^{*} \wedge \bar{V}_\mathcal F^{*}\rightarrow \Lambda^3V^*_\mathbb{C}.
\end{equation}
Consider the skew-symmetrization  
\begin{equation}\label{sk-1}\mathrm{sk} :  S^2V^*_\mathbb{C} \otimes V^*_\mathbb{C} \rightarrow V^*_\mathbb{C} \otimes \Lambda^2V^*_\mathbb{C}
\end{equation}
defined by (\ref{skew-def-1}).
From the exactness of the sequence 
\begin{equation}\label{map-complex}
0  \rightarrow  S^3V_{\mathbb{C}}^*
\rightarrow  S^2 V_{\mathbb{C}}^*\otimes V_\mathbb{C}^*\stackrel{\mathrm{sk}}{\longrightarrow}  V_\mathbb{C}^*\otimes \Lambda^2 
V_\mathbb{C}^*\stackrel{\partial}{\longrightarrow} \Lambda ^3 V_{\mathbb{C}}^*\rightarrow 0,
\end{equation}
the kernel of the map (\ref{map-c}) consists of linear combinations of elements of the form
 $u \otimes (v\wedge w) + v\otimes (u\wedge w)$, where 
$u, v, w\in V_{\mathbb{C}}^{*}$  and $v\wedge w$, $u\wedge w$ belong to $V^{*}_{\mathcal F}\wedge \bar{V}^{*}_{\mathcal F}$.  
The latter condition implies (for non-zero elements) that if $w\in V^{*}_{\mathcal F}$ then $u, v\in \bar{V}_{\mathcal F}^{*}$
and conversely,   if $w\in \bar{V}^{*}_{\mathcal F}$ then $u, v\in {V}_{\mathcal F}^{*}$.
We obtain that  the kernel 
of the map (\ref{map-c}) 
is given by the (faithful) image of 
\[  S^2V_\mathcal F^*\otimes \bar{V}_\mathcal F^* \oplus  S^2\bar{V}_\mathcal F^* \otimes V_\mathcal F^*\]
under the map  (\ref{sk-1}). Our claim follows.
\end{proof}

\subsection{$B_{n}$-K\"ahler  generalized connections}

The structure group of a $B_{n}$-generalized almost pseudo-Hermitian structure $(G, \mathcal{F})$ on $E$ is 
\[ H = \mathrm{O}(1) \times \mathrm{U}(k_1, l_{1})\times \mathrm{U}(k_2, l_{2})\]
where
\[ (k_1+ k_{2},l_1+ l_{2}) = \begin{cases} (m,m)\quad \mbox{if}\quad n=2m\\
(m+1,m) \quad\mbox{if}\quad n=2m+1.\end{cases}\]
So we need to calculate the first generalized prolongation of the Lie algebra 
\begin{align} 
\nonumber& \mathfrak{h} = \mathfrak{u}(k_1, l_{1})\oplus \mathfrak{u}(k_2, l_{2}) \subset \mathfrak{so}(2k_1, 2l_{1}) \oplus \mathfrak{so}(2k_{2}, 2l_{2})\subset \mathfrak{so}(V),\\
\nonumber&  V= \mathbb{R}^{n+1,n}= \mathbb{R}^{2(k_1+ k_{2}),2(l_1+ l_{2})} \oplus \mathbb{R}
\end{align}
where $\mathbb{R}^{2k_{1}, 2l_{1}} = V_{+}$, $\mathbb{R}^{2k_{2}, 2l_{2}} = V_{-}$ are the $\pm1$-eigenspaces of $G^{\mathrm{end}}$ on the orthogonal complement 
$\mathbb{R}^{2( k_{1} + k_{2}), 2( l_{1} + l_{2})}$  of 
the positive (if $n$ is even) or negative  (otherwise)  definite line
$\mathbb{R}$. We denote by  $V_{+, \mathcal F}\cong \mathbb{C}^{k_{1} + l_{1}}$ and $V_{-, \mathcal F}\cong \mathbb{C}^{k_{2} + l_{2}}$. 
the $i$-eigenspaces of $\mathcal F$ on $V_{\pm}$.

\begin{prop}\label{um_1um_2:prop} The generalized first prolongation of the Lie algebra 
\[  \mathfrak{h} = \mathfrak{u}(k_1, l_{1})\oplus \mathfrak{u}(k_2, l_{2}) \subset \mathfrak{so}(V)\] 
is given by  $\mathfrak{h}^{\langle 1\rangle }  =  \mathfrak{u}(k_1 , l_{1})^{\langle 1\rangle }  \oplus \mathfrak{u}(k_2, l_{2})^{\langle 1\rangle }$, where 
\begin{eqnarray*} 
 \mathfrak{u}(k_1, l_{1})^{\langle 1\rangle } &=& \{ \mathrm{Re} \left( \mathrm{sk} (\sigma) \right) \mid  \sigma \in S^2V_{+,\mathcal F}^* \otimes \bar{V}_{+,\mathcal F}^*\}\cong 
  S^2\mathbb{C}^{k_1+ l_{1}} \otimes \mathbb{C}^{k_1+ l_{1} }\\
 \mathfrak{u}(k_2, l_{2})^{\langle 1\rangle } &=& \{ \mathrm{Re} \left( \mathrm{sk} (\sigma) \right) \mid  \sigma \in S^2V_{-,\mathcal F}^* \otimes \bar{V}_{-,\mathcal F}^*\}\cong 
  S^2\mathbb{C}^{k_2+ l_{2}} \otimes \mathbb{C}^{k_2+ l_{2}}.
\end{eqnarray*}
\end{prop}   

\begin{proof} Using that  $\mathfrak{h}$ corresponds to the set of real points of 
\[ V_{+,\mathcal F}^*\wedge  \bar{V}_{+,\mathcal F}^*\oplus V_{-,\mathcal F}^*\wedge \bar{V}_{-,\mathcal F}^*\] 
under the identification of $\mathfrak{so}(V)\cong \Lambda^2V^*$, the claim easily follows from
an argument similar to 
 Proposition \ref{1stprol:prop}. 
\end{proof}

V. Cort\'es: vicente.cortes@math.uni-hamburg.de\

Department of Mathematics and Center for Mathematical Physics, University of Hamburg,  Bundesstrasse 55, D-20146, Hamburg, Germany.\\

L. David: liana.david@imar.ro\

Institute of Mathematics  `Simion Stoilow' of the Romanian Academy,   Calea Grivitei no.\ 21,  Sector 1, 010702, Bucharest, Romania.\\

M. Mirea: mireammi@yahoo.com\

Institute of Mathematics  `Simion Stoilow' of the Romanian Academy,   Calea Grivitei no.\ 21,  Sector 1, 010702, Bucharest, Romania.

\end{document}